\documentclass[oneside,english]{amsart}
\usepackage[T1]{fontenc}
\usepackage[latin9]{inputenc}
\usepackage{geometry}
\usepackage{amstext}
\usepackage{amsthm}
\usepackage{amssymb}
\usepackage{graphicx}

\makeatletter
\numberwithin{equation}{section}
\numberwithin{figure}{section}
\theoremstyle{plain}
\newtheorem{thm}{\protect\theoremname}
  \theoremstyle{remark}
  
  \theoremstyle{definition}
  \newtheorem{example}[thm]{\protect\examplename}
  \theoremstyle{plain}
  \newtheorem{cor}[thm]{\protect\corollaryname}
  \newtheorem{lem}[thm]{\protect\lemmaname}
\makeatother
\DeclareMathOperator{\rad}{rad}
\usepackage{babel}
  \providecommand{\corollaryname}{Corollary}
  \providecommand{\examplename}{Example}
  \providecommand{\remarkname}{Remark}
  \providecommand{\lemmaname}{Lemma}
\providecommand{\theoremname}{Theorem}

\begin{document}

\title{Special Values of $q$-Gamma Products}

\author{Tanay Wakhare$^{\dag}$}
\address{$^{\dag}$~University of Maryland, College Park, MD 20742, USA}
\email{twakhare@gmail.com}

\maketitle

\begin{abstract}
We consider products of $q$-gamma functions with rational arguments, and prove several $q$-generalizations of recent works concerning products of gamma functions. In particular, we consider products indexed by Dirichlet characters, and provide several new values for infinite products in terms of $e^{-\pi}$ and gamma functions.
\end{abstract}

\section{Introduction}
Over the past several years, significant attention has been paid to products of gamma functions with rational arguments. These products have unexpected number theoretic \cite{Straub} and group theoretic \cite{ShortGamma} connections, and provide researchers with convenient simplifications in many areas of pure mathematics. 

However, significantly less attention has been paid to the \textit{$q$-gamma function}. A natural generalization of the gamma function, the $q$-gamma function also admits special values in the $0<q<1$ regime which lead to beautiful new infinite product identities. These closed forms hint at Ramanujan-style modular transformations which can potentially lead to new special values for other $q$ generalizations of special functions.

In this paper, we generalize several well known identities for products of gamma functions by considering products indexed by non-trivial Dirichlet characters. By specializing these at values of $q$ besides $1$ we obtain several new values of infinite products in terms of $e^{-\pi}$ and gamma functions. This presents an extension of the program begun in \cite{Dilcher}, which generalized many classical product identities to identities involving Dirichlet characters. We also encounter in a natural way a variant of the cyclotomic polynomials defined by  
\begin{equation}\label{psi_def}
\Psi_n(x):=\prod_{d|n}({1-x^{d}})^{\mu(d)}.
\end{equation}
Note that while $\Psi_n$ is \textit{almost} a cyclotomic polynomial, this would require $x^{\frac{n}{d}}$ instead. We later completely characterize these polynomials through 
\begin{equation}
\Psi_n(x)=\Phi_{\rad n}(x)^{\mu(\rad n)},
\end{equation}
where $\Phi_n(x)$ denotes the $n$-th cyclotomic polynomial and $\rad n:=\prod_{p|n}p$ denotes the \textit{radical} of $n$. In Section \ref{section2} we prove the main $q$-generalization and provide several applications, while in Section \ref{section3} we study the functions $\Psi_n(x)$. We proceed formally, though all of our product manipulations can be made rigorous by considering regions of absolute convergence of the terms inside our products.

The prototype for our identities will be the simple product
\begin{equation}
\prod_{k=1}^\infty \left(1 - \frac{(-1)^k}{2k+1}\right) = \pi\frac{\sqrt{2}}{4}.
\end{equation}
The salient feature of this product is that it can be rewritten as $ \prod_{k=1}^\infty \left(1 +  \frac{f(k)}{k}\right),$ where $f(2)=f(4)=0$, $f(1)=1$, $f(3)=-1$, and $f(k) = f(k+4)$, i.e., $f$ is the unique non-principal Dirichlet character modulo $4$. This observation is the key feature of the manuscript \cite{Dilcher}, and will be extensively exploited here. 

\section{$q$-analog}\label{section2}
We define the gamma function $\Gamma$ in terms of the improper integral $\Gamma(z)= \int_0^\infty t^{z-1}e^{-t}dt$, which converges for $z\in \mathbb{C}$ with $\mathfrak{R}z > 0$. This is extended to all $z\in \mathbb{C}$ excluding the non-positive integers, by analytic continuation \cite{DLMF}. For $|q|<1, q\in \mathbb{C}$, we can then define the $q$-gamma function $\Gamma_q$ by
 $$ \Gamma_q(x)  : = (1-q)^{1-x} \frac{(q;q)_\infty}{ (q^{x};q)_\infty} = (1-q)^{1-x}\prod_{n\geq 0} \frac{1-q^{n+1}}{1-q^{n+x}},$$
where we have used the standard $q$-Pochhammer notation $(a;q)_0 = 1$ and $(a;q)_n := \prod_{k=0}^{n-1} (1-aq^k)$ for $n\in \mathbb{N}, n\geq1$ \cite{DLMF}. Note that $\Gamma_q(1)=1$, regardless of the value of $q$, and that $\Gamma_q \to \Gamma$ in the $q\to 1$ limit. We then have the following easy result, which appears to be absent from the literature:

\begin{thm}\label{thm1}
Let $\{\alpha_1,\ldots, \alpha_k\}$ and $\{\beta_1,\ldots, \beta_k\}$ be nonzero complex numbers such that $\sum_{i=1}^k \alpha_i = \sum_{i=1}^k \beta_i$. Then
\begin{equation}
\prod_{n\geq 0}  \prod_{j=1}^k  \frac{1-q^{n+\alpha_j}}{1-q^{n+\beta_j}} = \prod_{j=1}^k\frac{\Gamma_q(\beta_j)}{\Gamma_q(\alpha_j)}.
\end{equation}
\end{thm}
\begin{proof}
The proof is quite straightforward and depends solely on the definition of $\Gamma_q$. We have the string of equalities
\begin{align*}
\prod_{j=1}^k\frac{\Gamma_q(\beta_j)}{\Gamma_q(\alpha_j)} &= \prod_{j=1}^k \frac{(1-q)^{1-\beta_j}}{(1-q)^{1-\alpha_j}} \prod_{n\geq 0}  \frac{1-q^{n+1}}{1-q^{n+\beta_j}}  \frac{1-q^{n+\alpha_j}}{1-q^{n+1}} \\
& ={(1-q)^{\sum_{j=1}^k\alpha_j -\beta_j}} \prod_{n\geq 0} \prod_{j=1}^k \frac{1-q^{n+\alpha_j}}{1-q^{n+\beta_j}}.
\end{align*}
\end{proof}
In the $q\to1$ limit we recover a known result for infinite products over rational functions of $n$.
\begin{cor}\cite[Thm. 1]{Straub}\label{cor2}
Let $\{\alpha_1,\ldots, \alpha_k\}$ and $\{\beta_1,\ldots, \beta_k\}$ be nonzero complex numbers, non of which are negative integers, such that $\sum_{i=1}^k \alpha_i = \sum_{i=1}^k \beta_i$. Then
\begin{equation}\label{eq_cor2}
\prod_{n\geq 0}  \prod_{j=1}^k  \frac{n+\alpha_j}{n+\beta_j} = \prod_{j=1}^k\frac{\Gamma(\beta_j)}{\Gamma(\alpha_j)}.
\end{equation}
\end{cor}

The summation condition $\sum_{i=1}^k \alpha_i = \sum_{i=1}^k \beta_i$ in Theorem \ref{thm1} is not essential like in the classical gamma case -- it just ensures that the prefactors cancel correctly. However, the summation condition is necessary in the $q\to1$ limit for the product to converge. This simple theorem then allows us to generalize various products of gamma functions at rational numbers, with an interesting number theoretic twist.

\begin{thm}\label{coprime}
Let $\phi(n)$ denote Euler's totient function and $\Phi_n(x)$ denote the $n$-th cyclotomic polynomial. Then
$$ \prod_{k=1}^n \Gamma_q\left(\frac{k}{n}\right) = (1-q)^{\frac{n-1}{2}} \frac{(q;q)_\infty^n}{(q^{\frac{1}{n}};q^{\frac{1}{n}})_\infty}$$
and
$$\prod_{ \substack{ k=1 \\  (n,k)=1}}^n  \Gamma_q\left(\frac{k}{n}\right)  = (1-q)^{\frac{\phi(n)}{2}} \frac{ (q;q)_\infty^{\phi(n)}}{\prod_{k=1}^\infty \Phi_{\rad n}\left(q^{\frac{k}{n}}\right)^{\mu(\rad n)}}.$$
\end{thm}
\begin{proof}
The first equation follows from rewriting the product as
$$ \prod_{k=1}^n \Gamma_q\left(\frac{k}{n}\right) =    \prod_{k=1}^n (1-q)^{1-\frac{k}{n}} \frac{(q;q)_\infty}{ (q^{\frac{k}{n}};q)_\infty}$$
and noting the simple sum $\sum_{k=1}^n \left(1-\frac{k}{n}\right) = \frac{n-1}{2}$ and the dissection $$\prod_{k=1}^n (q^{\frac{k}{n}};q)_\infty = {(q^{\frac{1}{n}};q^{\frac{1}{n}})_\infty}. $$

The second identity follows the idea of \cite{Sandor}, which proved an analogous result for classical gamma functions. We consider the logarithm
$$F(n) := \sum_{k=1}^n \log \Gamma_q\left(\frac{k}{n}\right) = \frac{n-1}{2} \log (1-q) + n\log (q;q)_\infty - \sum_{k=1}^\infty \log (1-q^\frac{k}{n}) $$
and the auxiliary function 
$$G(n)  :=  \sum_{\substack{k=1 \\ (n,k)=1}}^n \log \Gamma_q\left(\frac{k}{n}\right).$$

From the M\"obius inversion formula, we know that any pair of functions $$f(n):=\sum_{k=1}^n \alpha_k, \thinspace\thinspace\thinspace\thinspace g(n):=\sum_{\substack{k=1\\ (n,k)=1}}^n \alpha_k$$ are related as $$g(n) = \sum_{d|n}\mu(d) f\left(\frac{n}{d}\right).$$ We can directly calculate this M\"obius convolution and exponentiate $G(n)$ to obtain our desired product. We use the results \cite{Apostol} (for $n\geq 2$) $$\sum_{d|n} \mu(d)=0, \thinspace\thinspace\thinspace\thinspace\sum_{d|n} \frac{\mu(d)}{d} = \frac{\phi(n)}{n}.$$ Thus
\begin{align*}
G(n) &= \frac{\log (1-q)}{2} \sum_{d|n} \left(\frac{n}{d}-1\right)\mu(d)  + \log (q;q)_\infty \sum_{d|n} \frac{n}{d} \mu(d) -  \sum_{d|n}\sum_{k=1}^\infty \log (1-q^\frac{kd}{n}) \mu(d) \\
&= \frac{\log (1-q)}{2} \phi(n) +  \log (q;q)_\infty \phi(n) - \sum_{k=1}^\infty \log   \prod_{d|n}\left(1-q^\frac{kd}{n}\right) ^{ \mu(d)}.
\end{align*}
Exponentiating $G(n)$ gives the formula $$\prod_{ \substack{ k=1 \\  (n,k)=1}}^n  \Gamma_q\left(\frac{k}{n}\right)  = (1-q)^{\frac{\phi(n)}{2}} \frac{ (q;q)_\infty^{\phi(n)}}{\prod_{k=1}^\infty \Psi_n\left(q^{\frac{k}{n}}\right)},$$
in terms of the modified cyclotomic polynomials $\Psi_n(x)$ defined by Equation \eqref{psi_def}. Applying Theorem \ref{psi_reduction} completes the proof.
\end{proof}

The next identity is an unusual $q$-generalization of the following very recent result, which is based on the limit $\frac{1-q^a}{1-q^b} \to \frac{b}{a} $ as $q\to 1$. For the statement of the next theorem, we use the von Mangoldt function $\Lambda(n)$, which is defined as 
\begin{equation}
\Lambda(n)= \begin{cases} 
\log p &\mbox{if } n=p^a \\ 
0 & \mbox{else} \\
\end{cases}.	 
\end{equation}
\begin{thm}\label{thm4}\cite{Dilcher}
Let $\chi$ denote a primitive non-principal Dirichlet character with conductor $k>1$. Then
$$\prod_{n=2}^\infty \left(1 - \chi(n) \frac{z}{n}\right) = \frac{(2\pi)^{\frac{\phi(k)}{2}}}{(1-z) e^{\frac{\Lambda(k)}{2}}} \prod_{\substack{j=1 \\ (j,k)=1}}^{k-1} \frac{1}{\Gamma\left(\frac{j-\chi(j)z}{k} \right)}.$$
\end{thm}

The advantage of this nonstandard $q$-analog is that we have special values for $\Gamma_q(x)$ when $q=e^{-n\pi}$, which lead to unexpected new product identities. 
\begin{thm}\label{thm5}
Let $\chi$ denote a primitive non-principal Dirichlet character with conductor $k>1$. Then
$$\prod_{n\geq 2}  \left(   \frac{1-q^{n-\chi(n)z}}{1-q^n} \right)  = \left(\frac{1-q}{1-q^{1-z}}\right)\prod_{j=1}^k\frac{\Gamma_{q^k}( \frac{j}{k})}{\Gamma_{q^k}\left( \frac{j-\chi(j) z}{k} \right)}.$$
\end{thm}
Though the left-hand side of this theorem obviously reduces to the correct left-hand side as in Theorem \ref{thm4}, the reduction of the right-hand side is not as obvious. In particular, the appearance of the von Mangoldt function necessitates that we consider at least two different cases, depending on whether the conductor is a prime power or not.
\begin{proof}
We make some informed substitutions into Theorem \ref{thm1}, so that we can rewrite the double product on the left hand side as a single product over $n$; this is based on a formal decomposition of the natural numbers into residue classes mod $k$ as $$\prod_{n=0}^\infty \prod_{j=1}^k \alpha_{kn+j} = \prod_{n=1}^\infty \alpha_n.$$ Apply Theorem \ref{thm1} with $\beta_j = \frac{j}{k}$ and $\alpha_j =\beta_j -\frac{\chi(j) z}{k}$. Due to the non-principality of $\chi$ we know that $\sum_{j=1}^k\chi(j) = 0$, so the summation condition is satisfied. We then obtain
\begin{align*}
\prod_{j=1}^k\frac{\Gamma_{q^k}( \frac{j}{k})}{\Gamma_{q^k}\left( \frac{j}{k} - \frac{\chi(j) z}{k} \right)} = \prod_{n\geq 0}  \prod_{j=1}^k  \frac{1-q^{kn+j - \chi(j)z}}{1-q^{kn+j} } = \left(\frac{1-q^{1-z}}{1-q}\right)\prod_{n\geq 2}  \left(   \frac{1-q^{n-\chi(n)z}}{1-q^n} \right).
\end{align*}
In the last equality, we exploited a character's periodicity mod $k$ and the fact $\chi(1)=1$. We also enforce that it is non-principal so that $\sum_{j=1}^k \chi(j) = 0$, and the symmetry condition is met. 
\end{proof}

By appealing to Theorem \ref{coprime} and noting that the term inside the product in Theorem \ref{thm5} is $1$ if $(j,k)>1$ we have the following full $q$-analog:
\begin{cor}\label{cor6}
Let $n\geq 2$ and let $\chi$ denote a primitive non-principal Dirichlet character with conductor $k>1$.  Let $\Phi_n(x):=\prod_{d|n}({1-x^{d}})^{\mu(d)}$. Then
$$\prod_{n\geq 2}  \left(   \frac{1-q^{n-\chi(n)z}}{1-q^n} \right)  = \frac{(1-q)(1-q^k)^{\frac{\phi(k)}{2}}}{1-q^{1-z}}      \frac{ (q^k;q^k)_\infty^{\phi(k)}}{\prod_{j=1}^\infty \Phi_{\rad k}\left(q^{j}\right)^{\mu\left(\rad k\right)}}      \prod_{\substack{j=1\\(j,k)=1}}^k\frac{1}{\Gamma_{q^k}\left( \frac{j-\chi(j) z}{k} \right)}.$$
\end{cor}

\begin{example}
Throughout, our special values for $\Gamma_q$ are pulled from \cite{Jackson}. Let $\chi(1)=1, \chi(3)=-1, \chi(2)=\chi(4)=0$ be the unique non-principal character mod 4. Also note that
$$ \Gamma_{q}\left(x+1\right) = \frac{1-q^{x}}{1-q}  \Gamma_{q}\left(x\right),$$
the functional relation for the $q$ gamma function. Apply this with $q\to q^4$ and $x = \frac{1-z}{4}$ so 
$$ \Gamma_{q^4}\left(\frac{5-z}{4}\right)= \frac{1-q^{1-z}}{1-q^4} \Gamma_{q^4} \left(\frac{1-z}{4}\right).$$ Applying Theorem \ref{thm5} with this particular character gives 
$$\prod_{n\geq 2}  \left(   \frac{1-q^{n-\chi(n)z}}{1-q^n} \right)  = \left(\frac{1-q}{1-q^4}\right)\frac{\Gamma_{q^4}( \frac14)\Gamma_{q^4}( \frac34)}{\Gamma_{q^4}\left( \frac{5- z}{4} \right)\Gamma_{q^4}\left( \frac{3+ z}{4} \right)}.$$
Now let $z \to 1$ and $q \to e^{-\pi}$ while noting $\Gamma_q(1)=1$. Therefore the denominator vanishes and 
\begin{equation}
\prod_{n\geq 2}  \left(   \frac{1-e^{ - n\pi+\chi(n) \pi }}{1-e^{-n\pi}} \right)  = \frac{e^{\frac{3\pi}{8}} (1-e^{-\pi})} {2^{\frac{23}{8}} \pi^{\frac32}} \Gamma^2\left(\frac14\right),
\end{equation}
where we've used the identity \cite{Jackson}
$$\Gamma_{e^{-4\pi}}\left(\frac14\right) \Gamma_{e^{-4\pi}}\left(\frac34\right) = \frac{e^{-\frac{29\pi}{8}} (e^{4\pi}-1)} {2^{\frac{23}{8}} \pi^{\frac32} } \Gamma^2\left(\frac14\right). $$

If instead we let $z\to -1$ and use the functional equation for the $q$-gamma function then we can write
$$\prod_{n\geq 2}  \left(   \frac{1-q^{n+\chi(n)}}{1-q^n} \right)  = \left(\frac{1-q}{1-q^2}\right)\frac{\Gamma_{q^4}( \frac14)\Gamma_{q^4}( \frac34)}{\Gamma_{q^4}^2\left( \frac{1}{2} \right)}.$$
Letting $q\to e^{-\pi}$ and using the special value
$$\Gamma_{e^{-4\pi}} \left(\frac12\right) = \frac{e^{-\frac74\pi} \sqrt{e^{4\pi}-1  } }{2^{\frac74}\pi^{\frac34}} \Gamma\left(\frac14\right)$$
lets us finally state the unexpected
\begin{equation}
\prod_{n\geq 2}  \left(   \frac{1-e^{ - n\pi-\chi(n) \pi }}{1-e^{-n\pi}} \right)  = \frac{2^{\frac58} e^{-\frac{\pi}{8}}}{ 1+e^{-\pi}}.
\end{equation}
\end{example}

\begin{example}
In the previous example, if we instead let $q\to e^{-2\pi}$ and exploit the reduction formulae \cite{Jackson}
\begin{align*}
\Gamma_{e^{-8\pi}} \left(\frac12\right) &= \frac{e^{-\frac72\pi} \sqrt{e^{8\pi}-1  } }{2^{\frac94}\pi^{\frac34} \sqrt{1+\sqrt{2}}} \Gamma\left(\frac14\right), \\
\Gamma_{e^{-8\pi}}\left(\frac14\right) \Gamma_{e^{-8\pi}}\left(\frac34\right) &= \frac{e^{-\frac{29\pi}{4}} (e^{8\pi}-1)} {16 \pi^{\frac32}\sqrt{1+\sqrt{2}} } \Gamma^2\left(\frac14\right),
\end{align*}
we obtain the higher order identities
\begin{align}
\prod_{n\geq 2}  \left(   \frac{1-e^{ - 2n\pi+2\chi(n) \pi }}{1-e^{-2n\pi}} \right)  &= \frac{e^{\frac{3\pi}{4}}(1-e^{-2\pi})} {16 \pi^{\frac32} } \Gamma^2\left(\frac14\right), \\
\prod_{n\geq 2}  \left(   \frac{1-e^{ - 2n\pi-2\chi(n) \pi }}{1-e^{-2n\pi}} \right)  &= \frac{\sqrt{2+2\sqrt{2}} e^{-\frac{\pi}{4}}}{ 1+e^{-2\pi}}.
\end{align}
\end{example}

\begin{example}
Jacobi and Legendre symbols are in fact special cases of non-principal characters, leading to simplified versions of the previous theorems. Given an odd prime $p$, the Legendre symbol is defined as 
\begin{equation}
(n|p) = \begin{cases} 
0 &\mbox{if } (n,p)>1 \\ 
1 & \mbox{if $n$ is a quadratic residue} \\
-1 & \mbox{if $n$ is not a quadratic residue}  
\end{cases}.	
\end{equation}
This is a non-principal Dirichlet character with period $p$, so in particular we can apply Theorem \ref{thm5} and Corollary \ref{cor6} to it.
\end{example}

\section{Modified Cyclotomic Polynomials}\label{section3} 
In this section we consider the modified cyclotomic polynomial 
\begin{equation*}
\Psi_n(x):=\prod_{d|n}\left({1-x^{d}}\right)^{\mu(d)}.
\end{equation*}
Because of their close connection to the ubiquitous cyclotomic polynomials, they are deserving of independent consideration. In fact, we present a complete characterization in terms of cyclotomic polynomials, which allows us to simplify the statement of Theorem \ref{coprime}. First, by applying M\"obius inversion to their definition, we note $$(1-x^n)^{\mu(n)} = \prod_{d|n}\Psi_d(x)^{\mu(d)}. $$ By considering simple cases such as $\Psi_2(x) = \frac{1-x}{1-x^2}$, we see that the $\Phi_n$ are not even polynomials, but are instead rational functions. 
\begin{lem}
We have the following reduction formulae:
\begin{equation}
\Psi_{p^k n}(x) = \begin{cases} 
\Psi_n(x) & \mbox{if } p|n \\ 
\frac{\Psi_n(x)}{\Psi_n(x^p)} & \mbox{if }p \nmid n  
\end{cases}.	
\end{equation}
\end{lem}
\begin{proof}
First, assume $p|n$. Then $$\Psi_{pn}(x) = \prod_{d|pn}(1-x^d)^{\mu(d)} =\prod_{d|n}(1-x^d)^{\mu(d)} \prod_{\substack{d|pn \\ d\nmid n}}(1-x^d)^{\mu(d)} = \Psi_n(x),  $$
since every term in the second product has $p^2|d$ so $\mu(d)=0$. Iterating this $k$ times proves the first case. Now assume $p\nmid n$. Then $$\Psi_{pn}(x) = \prod_{d|pn}(1-x^d)^{\mu(d)} =\prod_{d|n}(1-x^d)^{\mu(d)} \prod_{d|n}(1-x^{pd})^{\mu(pd)} = \Psi_n(x)  \prod_{d|n}(1-x^{pd})^{-\mu(d)}. $$ To address the full case of $p^kn$ where $(n, p) = 1$, we iterate our first result $k-1$ times and then apply our reduction: 
$$\Psi_{p^k n}(x)  = \Psi_{p^{k-1} n}(x) = \cdots = \Psi_{p n}(x)= \frac{\Psi_n(x)}{\Psi_n(x^p)}.$$
\end{proof}
From this lemma, we can deduce a complete characterization of $\Psi_n(x)$.
\begin{thm}\label{psi_reduction}
Let $\Phi_n(x)$ denote the $n$-th cyclotomic polynomial. Then  $$\Psi_{ n}(x) = \Phi_{\rad n}(x)^{\mu(\rad n)}.$$
\end{thm}
\begin{proof}
From the previous Lemma, we know $\Psi_n(x) = \Psi_{\rad n}(x)$, where $\rad n$ is squareefree. For squarefree $n$, we have the reduction $$\mu\left(\frac{n}{d}\right) = \mu\left(n\right)\mu\left({d}\right).$$ This is the key aspect of the proof -- since $n$ is squarefree, we only care about the parity of its number of prime factors; we can then consider four cases, based on the parity of the number of prime factors of $n$ and $d$. Manually checking all four cases shows that the equality holds.

Now we directly examine the definition of $\Psi_n(x)$ for squarefree $n:$
$$\Psi_n(x)=\prod_{d|n}\left({1-x^{d}}\right)^{\mu(d)}=\prod_{d|n}\left({1-x^{\frac{n}{d}}}\right)^{\mu\left(\frac{n}{d}\right)}=\prod_{d|n}\left({1-x^{\frac{n}{d}}}\right)^{\mu\left({d}\right)\mu\left({n}\right)} = \Phi_n(x)^{\mu(n)}.$$
Reducing $\Psi_n$ to $\Psi_{\rad n}$ and then rewriting in terms of cyclotomic polynomials completes the proof.
\end{proof}

\section{Acknowledgements}
I can never say this enough -- to Christophe Vignat, who's mentored me from an ignorant high schooler to a slightly less ignorant college student -- thank you, thank you, thank you! Additionally, I'd like to thank Karl Dilcher for first introducing me to this topic and very carefully reading the final draft of this manuscript.

\end{document}